\documentclass[11pt]{article}

\usepackage{hyperref}

\usepackage{amsmath,amssymb,amsthm,setspace,graphicx}
\usepackage[text={6.5in,8.4in}, top=1.3in]{geometry}

\newtheorem{theorem}{Theorem}[section]
\newtheorem{proposition}[theorem]{Proposition}
\newtheorem{lemma}[theorem]{Lemma}

\newtheorem{conjecture}[theorem]{Conjecture}
\newtheorem{problem}[theorem]{Problem}
\newtheorem{definition}[theorem]{Definition}

\newcommand{\T}{{\cal T}}
\newcommand{\R}{{\cal R}}
\newcommand{\B}{{\cal B}}
\newcommand{\F}{{\cal F}}
\newcommand{\A}{{\cal A}}

\begin{document}

\title{Dominant tournament families
	\thanks{This research was supported by the Israel Science Foundation (grant No. 1082/16).}}

\author{
Raphael Yuster
\thanks{Department of Mathematics, University of Haifa, Haifa 3498838, Israel.
	Email: raphael.yuster@gmail.com}
}

\date{}

\maketitle

\setcounter{page}{1}

\begin{abstract}

For a tournament $H$ with $h$ vertices, its typical density is $h!2^{-\binom{h}{2}}/aut(H)$,
i.e. this is the expected density of $H$ in a random tournament.
A family $\F$ of $h$-vertex tournaments is {\em dominant} if for all sufficiently large $n$, there exists an $n$-vertex tournament $G$ such that the density of each element of $\F$ in $G$ is larger than its typical
density by a constant factor. Characterizing all dominant families is challenging already for small $h$. Here we characterize several large dominant families for every $h$. In particular, we prove the following for all $h$ sufficiently large: (i) For all tournaments $H^*$ with at least $5\log h$ vertices, the family of all $h$-vertex tournaments that contain $H^*$ as a subgraph is dominant. (ii) The family of all $h$-vertex tournaments whose minimum feedback arc set size is at most $\frac{1}{2}\binom{h}{2}-h^{3/2}\sqrt{\ln h}$ is dominant.
For small $h$, we construct a dominant family of $6$ (i.e. $50\%$ of the) tournaments on $5$ vertices and
dominant families of size larger than $40\%$ for $h=6,7,8,9$.
For all $h$, we provide an explicit construction of a dominant family which is conjectured to obtain an absolute constant fraction of the tournaments on $h$ vertices. Some additional intriguing open problems are presented.

\vspace*{3mm}
\noindent
{\bf AMS subject classifications:} 05C20, 05C35\\
{\bf Keywords:} tournament; density

\end{abstract}

\section{Introduction}

All graphs in this paper are finite and simple. Our main objects of study are {\em tournaments}, namely orientations of the complete graph.
The {\em density} of a tournament $H$ with $h$ vertices in a larger tournament $G$ is the probability
$d_H(G)$ that a randomly chosen set of $h$ vertices of $G$ induces a tournament that is isomorphic to $H$ (i.e. an $H$-copy in
$G$). Stated otherwise, if $c_H(G)$ denotes the number of $H$-copies in an $n$-vertex tournament $G$, then $d_H(G)=c_H(G)/\binom{n}{h}$.

There are several papers that consider possible densities of a given tournament in larger tournaments
\cite{Bucic2019,Chan2019,Chung1991,Coregliano2017,Coregliano2019,Hancock2019,Linial2016}.
Broadly speaking, there are a few designated regimes of interest.
The maximum density of $H$, denoted by $d_{max}(H)$ is the limsup of the sequence whose $n$'th element is
the maximum possible value of $d_H(G)$ ranging over $n$-vertex tournaments $G$.
The maximum density is sometimes called the {\em inducibility} of $H$ \cite{Linial2016}.
Clearly $d_{max}(H)=1$ if and only if $H=T_h$ is the transitive tournament on $h$ vertices.
Determining $d_{max}(H)$ for some $H$ may be quite challenging; for some small $H$, flag algebra techniques
are useful \cite{Coregliano2019,Coregliano2017,Linial2016,Razborov2007}.
One can similarly consider the minimum density of $H$ denoted by $d_{min}(H)$, but of course
$d_{min}(H)=0$ unless $H=T_h$.
For the latter, $d_{min}(T_h)$ is the liminf of the sequence whose $n$'th element is
the minimum possible value of $d_{T_h}(G)$ ranging over $n$-vertex tournaments $G$.
The {\em typical density}, denoted by $d(H)$ is the expected density of $H$ is a random tournament.
By a random tournament we mean, as usual, the probability space of $n$-vertex tournaments where the direction of each edge is chosen independently and uniformly at random.
Observe that $d(H)$ is independent of $n$ and is easy to compute.
The probability of a labeled random $h$-vertex tournament to
be isomorphic to a labeled copy of $H$ is $2^{-\binom{h}{2}}$. Hence, $d(H)=h!2^{-\binom{h}{2}}/aut(H)$ where
$aut(H)$ is the size of the automorphism group of $H$. In particular, $d(T_h)=h!2^{-\binom{h}{2}}$.
The typical density plays an important role in the study of quasi-random tournaments
\cite{Bucic2019,Chung1991,Coregliano2019,Hancock2019}.

By their definitions, we have that $d_{min}(H) \le d(H) \le d_{max}(H)$ for every $H$.
There are a few tournaments where one of the inequalities is an equality.
For the transitive tournament $T_h$ 
it is well-known that $d_{min}(T_h) = d(T_h)$ (see Exercise 10.44(b) of \cite{Lovasz2007}).
There are a few sporadic cases where $d_{max}(H) = d(H)$. This is easily shown to hold for $H=C_3$,
the directed triangle, but it is also known to hold for the tournament on $5$ vertices $H_5^8$ of Figure \ref{f:tour-5} as proved by Coregliano et al. \cite{Coregliano2019} (there called $T_5^8$). It is known that all tournaments on four vertices have $d_{max}(H) > d(H)$ as well
as all tournaments on at least $7$ vertices \cite{Bucic2019}.

Let $\T_h$ denote the set of all tournaments on $h$ vertices.
So on the one hand, for a given $H \in \T_h$ (except for the few sporadic cases where $d_{max}(H)=d(H)$ discussed above),
one can construct arbitrarily large tournaments $G$ in which $d_H(G)$ is significantly larger than the typical density $d(H)$,
but certainly no such $G$ can be universal for all elements of $\T_h$ since clearly for any $G$ we have
$$
1 = \sum_{H \in \T_h} d(H) = \sum_{H \in \T_h} d_H(G)\;.
$$
So, the natural question that emerges is, to what extent can a significant subset $\F \subset \T_h$ have the property
that there are arbitrarily large tournaments $G$ that are universal for all elements of $\F$.

\begin{definition}
A set $\F \subset \T_h$ is {\em dominant} if there exists $\beta > 0$ such that for all sufficiently
large $n$, there exists an $n$-vertex tournament $G$ for which $d_H(G) \ge (1+\beta)d(H)$ for all $H \in \F$.
\end{definition}
Trivially, all singletons (except for the sporadic cases discussed above where $d_{max}(H)=d(H)$) are dominant, but
we are of course interested with the existence of large dominant $\F$.
Clearly, if one can characterize all maximal dominant $\F$ then this would characterize all dominant $\F$, but at
present this seems like a problem beyond our reach (we do not even have an exact formula for the number of elements of $\T_h$). A more realistic goal is to determine large $\F$ that can be explicitly characterized in the sense that the members
of $\F$ are exactly the ones that satisfy some natural property (namely, given a tournament $H$, one can deterministically check whether $H$ satisfies the property). This is indeed what we do in this paper for a
few very natural properties.

In Section \ref{sec:bias}, we prove that for all sufficiently large $h$, the family of tournaments whose minimum feedback arc set size is at most
$\frac{1}{2}\binom{h}{2}-h^{3/2}\sqrt{\log h}$ \footnote{Unless stated otherwise, all logarithms are in base $2$.} is a dominant family. We note that this result cannot be improved by much as it is well-known that
the minimum feedback arc set size of {\em every} $h$-vertex tournament is at most $\frac{1}{2}\binom{h}{2}-
\Theta(h^{3/2})$ \cite{Spencer1971}.
Our main tool in the proof is the notion of the bias polynomial (a notion defined in Section \ref{sec:bias}).
We also prove that the subset of all $h$-vertex tournaments whose bias polynomial has a local minimum at $0$, is dominant.
We show that for some small $h$, this subset is of significant size. For example, for each $h=6,7,8,9$ more than 40\% of the tournaments on $h$ vertices are of this type, and half of the tournaments on $5$ vertices are of this type.
We conjecture that for all $h$, the fraction of such tournaments out of all $h$-vertex tournaments is at least a positive constant independent of $h$.

In Section \ref{sec:H*}, we prove that for all sufficiently large $h$, if $H^*$ is a tournament with at least
$5\log h$ vertices, then the family of all elements of $\T_h$ that contain an $H^*$-copy, is dominant.
Again, this result cannot be improved by much as it is well-known \cite{Stearns1959}
that every element of $\T_h$ contains $T_{\lceil {\log h} \rceil}$.

In section \ref{sec:concluding}, we discuss a few open problems and conjectures related to dominant families. Solving some of these problems may be challenging.

\section{The bias polynomial and dominant families}\label{sec:bias}

\subsection{The bias polynomial}

We define a probability space on labeled $n$-vertex tournaments that generalizes
the standard uniform probability space (the random tournament model).
Consider tournaments with labeled vertices $[n]=\{1,\ldots,n\}$ and let $p \in [0,1]$.
If $i < j$ then make $(i,j)$ an edge with probability $p$ (so $(j,i)$ is an edge with probability $1-p$) where all $\binom{n}{2}$ choices are independent.
Denote the resulting probability space by $T(n,p)$ and observe that $T(n,\frac{1}{2})$ is the usual notion
of a random tournament. We note that there are other models of random graphs where the probability of an edge depends on the order of vertex labels (see, e.g., \cite{Alon2008}).

Given $G \sim T(n,p)$, define the typical density of $H$ in $G$, denoted by $d(H,p)$,
to be the expected density of $H$ in $G$. Notice that $d(H)=d(H,\frac{1}{2})$.
Using Chebyshev's inequality, it is easy
to prove that $S=\{H \in \T_h\,|\, d(H,p) > d(H)\}$ is dominant (see the proof of Lemma \ref{l:concentration} below).
However, recall that we would like to obtain explicit constructions of large dominant sets and for this we need to pinpoint some explicit range of $p$ that ensures that $S$ is large. To this end, it is beneficial to observe that $d(H,p)$ is, in fact, a polynomial in $p$.
Indeed, each order of the vertices of $H$ corresponds to a term in $d(H,p)$ of the form $p^k(1-p)^{\binom{h}{2}-k}$
where $k$ is the number of edges of $H$ pointing from a lower ordered vertex to a higher one.
So, for instance, for $H=T_3$ we have that $d(T_3,p)=p^3+(1-p)^3+2p^2(1-p)+2p(1-p)^2=1-p+p^2$
while for $H=C_3$ we have $d(C_3,p)=p^2(1-p)+p(1-p)^2=p-p^2$. Since, by symmetry, $d(H,p)=d(H,1-p)$ it is more convenient to work with the following definition.
\begin{definition}
	The {\em bias polynomial} of $H$ is $B(H,x)=d(H,x+\frac{1}{2})$.
\end{definition}
The following simple lemma lists some obvious properties of the bias polynomial.
\begin{lemma}\label{l:bias-properties}
	Let $B(H,x)$ be the bias polynomial of a tournament $H$ with $h$ vertices.
\begin{enumerate}
	\item $B(H,x)$ is an even polynomial. Equivalently, each term of $B(H,x)$ is a constant multiple of $x$ to an even power.
	\item $B(H,0)=d(H)$, $B(T_h,\pm \frac{1}{2})=1$ and otherwise $B(H,\pm \frac{1}{2})=0$.
	\item $0$ is a local extremum of $B(H,x)$. It is a local minimum if and only if the coefficient of the lowest order term of $B(H,x)-d(H)$ is positive.
	\item $\sum_{H \in \T_h} B(H,x) = 1$.
\end{enumerate}
\end{lemma}
\begin{proof}
	Property 1 follows since $B(H,x)=d(H,x+\frac{1}{2})=d(H,\frac{1}{2}-x)=B(H,-x)$. Property 2 follows since
	$B(H,0)=d(H,\frac{1}{2})=d(H)$. Property 3 follows since $B(H,x)$ is an even polynomial and the condition for local minimum follows since this is the case when the derivative at zero changes sign from negative to positive.
	Property 4 follows from the fact that for every $0 \le p \le 1$, $\sum_{H \in \T_h}d(H,p)=1$.
\end{proof}
Let $\F(h,x)=\{H \in \T_h\,|\, B(H,x) > d(H)\}$. While $\F(h,0)=\emptyset$ and $F(h,\pm \frac{1}{2})=\{T_h\}$,
we will prove that for certain $x=x(h)$, $\F(h,x)$ is large. For this to be of use, we need the following.
\begin{lemma}\label{l:concentration}
For every $x \in (0,\frac{1}{2})$, $\F(h,x)$ is dominant.
\end{lemma}
\begin{proof}
	Fix $0 < x < \frac{1}{2}$. Let
	$$
	\beta = \min_{H \in \F(h,x)} \frac{B(H,x)}{d(H)}-1\;.
	$$
	Observe that $\beta > 0$ since by the definition of $\F(h,x)$ we have $B(H,x) > d(H)$ for every $H \in \F(h,x)$.
	We prove that for all sufficiently large $n$, there is an $n$-vertex tournament $G$ such that $d_H(G) \ge (1+\beta/2)d(H)$ holds for all $H \in \F(h,x)$, thus obtaining that $\F(h,x)$ is dominant.
	
	Let $p=x+\frac{1}{2}$ and consider $G \sim \T(n,p)$. Let $H \in \F(h,x)$ and notice that
	$B(H,x)=d(H,p)$ is the expected density of $H$ in $G$.
	Recall that $c_H(G)$ denotes the number of $H$-copies in $G$.
	So, the expected value of $c_H(G)$ is $\binom{n}{h}B(H,x)=\Theta(n^h)$. We may consider each $h$-set of vertices of $G$ as an indicator random variable for the event that the corresponding $h$-set induces a copy of $H$, thus $c_H(G)$ is the sum of these $\binom{n}{h}$ variables, each with success probability $B(H,x)$.
	But also notice that two indicator variables corresponding to disjoint $h$-sets are independent.
	Hence, the variance of $c_H(G)$ is only $O(n^{2h-1})$. By the second moment method (see \cite{Alon2004}),
	the probability that $c_H(G)$ is smaller than its expected value by more than $\binom{n}{h}\frac{\beta}{2}d(H)$ is $O(n^{-1})$.
	Since $n$ is chosen sufficiently large, we may assume that $n \gg |\F(h,x)|$.
	Hence there exists an $n$-vertex tournament $G$ such that for all $H \in \F(h,x)$ it holds that
	$c_H(G) \ge \binom{n}{h}B(H,x)-\binom{n}{h}\frac{\beta}{2}d(H)$ 
	and equivalently $d_H(G) \ge B(H,x)-\frac{\beta}{2}d(H)$.
	Finally, notice that by the definition of $\beta$,
	$$
	d_H(G) \ge B(H,x)-\frac{\beta}{2}d(H) \ge d(H) + \frac{\beta}{2}d(H) = \left(1+\frac{\beta}{2}\right)d(H)\;.
	$$
\end{proof}

\subsection{Minimum feedback arc set and dominant families}

For a tournament $H$, a {\em feedback arc set} of $H$ is a set of edges covering every directed cycle.
Equivalently, it is a spanning subgraph of $H$ whose complement is acyclic.
Let $a(H)$ denote the cardinality of a smallest feedback arc set of $H$.
While it is straightforward that $a(H) \le  \frac{1}{2}\binom{h}{2}$ and that $a(H)=0$ if and only if $H=T_h$,
determining the precise value is NP-Hard in general \cite{Alon2006}. 
Spencer \cite{Spencer1971}, improving earlier results of
Erd\H{o}s and Moon \cite{Erdos1965}, proved that $a(H) \le \frac{1}{2}\binom{h}{2} - \Theta(h^{3/2})$.
We will prove that the set of all tournaments whose $a(H)$ value is slightly below this upper bound
is dominant.

Let $\A(h,t)$ denote the set of all tournaments having $a(H) \le \frac{1}{2}\binom{h}{2}-t$.
\begin{theorem}
	$\A(h,h^{3/2}\sqrt{\ln h})$ is dominant for all $h \ge 30$.
\end{theorem}
\begin{proof}
	We will prove that for all $h \ge 30$ it holds that $\A(h,h^{3/2}\sqrt{\ln h}) \subseteq \F(h,(\ln h/h)^{1/2})$ and hence the result will follow by Lemma \ref{l:concentration}.
	Let $x=(\ln h/h)^{1/2}$ and let $H \in \A(h,h^{3/2}\sqrt{\ln h})$.
	We must prove that $H \in \F(h,x)$, namely that $B(H,x) > d(H)$.
	Recalling that $d(H)=h!2^{-\binom{h}{2}}/aut(H)$, we must prove that $B(H,x) > h!2^{-\binom{h}{2}}/aut(H)$.
	
	Let the vertices of $H$ be labeled with $[h]=\{1,\ldots,h\}$. For a permutation $\pi \in S_h$, let
	$f(\pi)$ (the ``forward'' edges) denote the number of edges $(u,v)$ of $H$ with $\pi(u) < \pi(v)$ and let $b(\pi)=\binom{h}{2}-f(\pi)$ be the ``backward'' edges. Then we have that
	\begin{equation}\label{e:1}
	B(H,x) = \frac{1}{aut(H)}\sum_{\pi \in S_h} \left(\frac{1}{2}+x\right)^{f(\pi)}\left(\frac{1}{2}-x\right)^{b(\pi)}\;.
	\end{equation}
	So it suffices to prove that
	$$
	\sum_{\pi \in S_h} \left(\frac{1}{2}+x\right)^{f(\pi)}\left(\frac{1}{2}-x\right)^{b(\pi)} > h!2^{-\binom{h}{2}}\;.
	$$
	There are $h!$ terms on the left-hand side of the last inequality but some (in fact, most) of them are smaller than $2^{-\binom{h}{2}}$ as it is likely that for many permutations $\pi$ it holds that $f(\pi)$ and $b(\pi)$
	are very close, or $b(\pi)$ is larger than $f(\pi)$. But, on the other hand, we do know that for some permutation, $f(\pi)$ is considerably larger than $b(\pi)$. Indeed, since $H \in \A(h,h^{3/2}\sqrt{\ln h})$, there is
	a minimum feedback arc set of $H$ of size at most $\frac{1}{2}\binom{h}{2}-h^{3/2}\sqrt{\ln h}$.
	But recall that this means that there is an acyclic spanning subgraph of $H$ with at least
	$\frac{1}{2}\binom{h}{2}+h^{3/2}\sqrt{\ln h}$ edges. As each acyclic digraph has an ordering $\pi$
	of its vertices where all edges of the digraph are forward, we have that there exists $\pi_0$ such that
	$f(\pi_0) \ge \frac{1}{2}\binom{h}{2}+h^{3/2}\sqrt{\ln h}$ and consequently
	$b(\pi_0) \le \frac{1}{2}\binom{h}{2}-h^{3/2}\sqrt{\ln h}$. It therefore suffices to prove that
	$$
	\left(\frac{1}{2}+x\right)^{f(\pi_0)}\left(\frac{1}{2}-x\right)^{b(\pi_0)} > h!2^{-\binom{h}{2}}
	$$
	or equivalently that
	$$
	\left(1+2x\right)^{f(\pi_0)-b(\pi_0)}\left(1-4x^2\right)^{b(\pi_0)} > h!\;.
	$$
	Indeed, this holds since
	\begin{align*}
	& \left(1+2x\right)^{f(\pi_0)-b(\pi_0)}\left(1-4x^2\right)^{b(\pi_0)}	\\
	> & \left(1+\frac{2\sqrt{\ln h}}{\sqrt{h}}\right)^{2h^{3/2}\sqrt{\ln h}}\left(1-\frac{4\ln h}{h}\right)^{h^2/4}	\\
	> & e^{-2h\ln h}e^{3h\ln h} = h^h
	\end{align*}
	where the last inequality holds for all $h \ge 30$.	
\end{proof}

\begin{table}[t]
	\begin{center}
		\renewcommand{\arraystretch}{1.5}
		\begin{tabular}{|c|c|c|}
			\hline 
			tournament & bias polynomial & in $\B_h$\\ 
			\hline 
			$T_4$ & $\frac{3}{8}+2x^2+2x^4$ & $\surd$ \\ 
			\hline 
			$C_4$ &  $\frac{3}{8}-2x^2+2x^4$ & \\ 
			\hline 
			$D$ & $\frac{1}{8}-2x^4$ & \\ 
			\hline 
			$D^t$ & $\frac{1}{8}-2x^4$ & \\ 
			\hline 
		\end{tabular}
	\end{center}
	\caption{Tournaments on four vertices and their bias polynomials.}\label{table:1}
\end{table} 

It is important to stress that $\A(h,h^{3/2}\sqrt{\ln h})$, while large, is {\em not} a constant
proportion of the family $\T_h$, as proved by Spencer \cite{Spencer1980} and de la Vega \cite{Vega1983}.
But on the other hand $\A(h,h^{3/2}\sqrt{\ln h})$ does contain, say, quasi-random tournaments.
Indeed, by one of the equivalent notions of quasi-random tournaments proved by Chung and Graham \cite{Chung1991}, there is a quasi-random sequence of tournaments $\{H_h\}$ where $H_h$ has $h$ vertices such that
$H_h \in \A(h,h^{3/2}\sqrt{\ln h})$.

\begin{figure}[t]
	\includegraphics[scale=0.6,trim=20 200 160 45, clip]{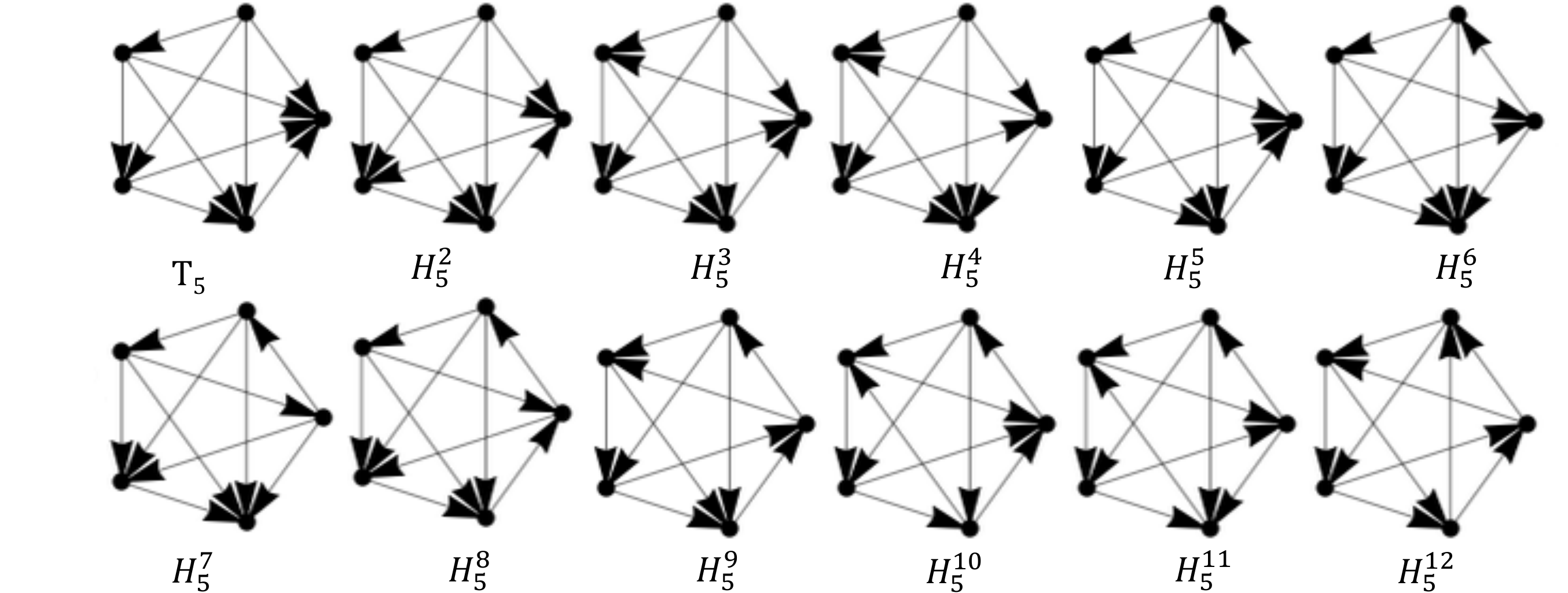}
	\caption{The tournaments on $5$ vertices.}
	\label{f:tour-5}
\end{figure}

\subsection{The bias subset}
Property 3 of Lemma \ref{l:bias-properties} states that we can partition $\T_h$ into two subsets: those tournaments $H$ for which $0$ is a local minimum of $B(H,x)$ and those for which $0$ is a local maximum of $B(H,x)$.
\begin{definition}
	The {\em bias subset} $\B_h \subset T_h$ consists of the tournaments $H \in \T_h$ for which $0$ is a local minimum of $B(H,x)$.
\end{definition}
For example, it is easy to verify that $B(T_3,x)=\frac{3}{4}+x^2$ while $B(C_3,x)=\frac{1}{4}-x^2$.
Hence, $\B_3 = \{T_3\}$.
The following is a corollary of Lemma \ref{l:concentration}.
\begin{proposition}\label{prop:concentration-bh}
$\B_h$ is dominant.
\end{proposition}
\begin{proof}
	For each $H \in \B_h$, let $\alpha_H >0$ be the largest real such that $B(H,x)$ is monotone increasing in $(0,\alpha_H)$. Such an interval exists since $0$ is a local minimum of $B(H,x)$.
	Notice that if $H \neq T_H$ then it must be that $0 < \alpha_H \le \frac{1}{2}$ since by Lemma \ref{l:bias-properties},
	$B(H,\frac{1}{2})=0$ and $B(H,0)=d(H) > 0$. If $H=T_h$ then it may be that $\alpha_H > \frac{1}{2}$ (in fact, it
	may be infinity) so if this occurs, just redefine $\alpha_{T_H}=\frac{1}{2}$.
	Now define $\alpha_h = \min \{\frac{1}{2}\alpha_H\,|\, H \in \B_h\}$. As $\alpha_h$ is a minimum of a finite set of positive reals, each no larger than $\frac{1}{4}$, we have that $0 < \alpha_h \le \frac{1}{4}$.
	By Lemma \ref{l:concentration}, $\F(h,\alpha_h)$ is dominant. As $\B_h \subseteq \F(h,\alpha_h)$, the proposition follows.
\end{proof}

Note that $\B_h$ is explicitly constructed, as for each tournament $H$ one merely needs to compute
the bias polynomial $B(H,x)$ as given in (\ref{e:1}) and check whether the coefficient of the lowest order term of $B(H,x)-d(H)=B(H,x)-B(H,0)$ is positive. In Tables \ref{table:1} and \ref{table:2} we list the bias polynomials of $\T_4$ and $\T_5$ respectively.  In particular, we obtain that $\B_5=\{T_5,H_5^2,H_5^3,H_5^4,H_5^6,H_5^7\}$
which is half of the total of $12$ tournaments on $5$ vertices.
In Table \ref{table:3} we list for all $3 \le h \le 9$ the size of $\B_h$ and the ratio of
$\B_h$ and $\T_h$. In particular, we have that $|\B_9|=79229$ which constitutes more than $41\%$ of the total number of tournaments on $9$ vertices \footnote{Source code of our program is available at \url{https://www.dropbox.com/s/y9zovepfr1hg1nt/dominant-tour.zip?dl=0}}. The following conjecture, if true, will give a dominant subset that is at least an absolute constant fraction of $\T_h$.
\begin{conjecture}\label{conj:bh}
	There exists an absolute constant $c > 0$ such that for all $h \ge 3$, $|\B_h| \ge c|\T_h|$.
\end{conjecture}

\begin{table}[t]
	\begin{center}
		\renewcommand{\arraystretch}{1.5}
		\begin{tabular}{|c|c|c|}
			\hline 
			tournament & bias polynomial & in $\B_h$ \\ 
			\hline 
			$T_5$ & $\frac{15}{128}+\frac{25}{16}x^2+6x^4+7x^6+2x^8$ & $\surd$\\ 
			\hline 
			$H_5^2$ &  $\frac{5}{128}+\frac{5}{16}x^2-\frac{1}{2}x^4-5x^6-2x^8$ &  $\surd$ \\ 
			\hline 
			$H_5^3$ &  $\frac{15}{128}+\frac{5}{16}x^2   -4x^4   +3x^6+2x^8  $ &  $\surd$\\ 
			\hline 
			$H_5^4$ &  $\frac{5}{128}+\frac{5}{16}x^2-\frac{1}{2}x^4-5x^6-2x^8$ & $\surd$\\ 
			\hline 
			$H_5^5$ &  $\frac{15}{128}-\frac{5}{16} x^2   +\frac{1}{2}x^4   -3x^6-6x^8 $ & \\ 
			\hline 
			$H_5^6$ &  $\frac{15}{128}+\frac{5}{16}x^2-4x^4+3x^6+2x^8$ & $\surd$\\ 
			\hline 
			$H_5^7$ &  $\frac{5}{128}+\frac{5}{16}x^2-\frac{1}{2}x^4-5x^6-2x^8$ & $\surd$\\ 
			\hline 
			$H_5^8$ &  $\frac{15}{128}-\frac{5}{16}x^2   -\frac{5}{2}x^4   +5x^6+10x^8   $ & \\ 
			\hline 
			$H_5^9$ &  $\frac{15}{128}-\frac{15}{16}x^2   +2x^4   -x^6+2x^8   $ & \\ 
			\hline 
			$H_5^{10}$ &  $\frac{5}{128}-\frac{5}{16}x^2   +x^4   -3x^6+6x^8  $ & \\ 
			\hline 
			$H_5^{11}$ &  $\frac{15}{128}-\frac{15}{16}x^2   +x^4   +7x^6-14x^8   $ & \\ 
			\hline 
			$H_5^{12}$ &  $\frac{3}{128}-\frac{5}{16}x^2   +\frac{3}{2}x^4   -3x^6+2x^8   $ & \\ 
			\hline 
		\end{tabular}
	\end{center}
	\caption{Tournaments on five vertices and their bias polynomials.}\label{table:2}
\end{table} 

\begin{table}[ht]
	\begin{center}
		\renewcommand{\arraystretch}{1.5}
		\begin{tabular}{|c||c|c|c|c|c|c|c|}
			\hline
			$h$ & $3$ & $4$ & $5$ & $6$ & $7$ & $8$ & $9$ \\
			\hline
			$|\T_h|$ & $2$ & $4$ & $12$ & $56$ & $456$ & $6880$ & $191536$ \\
			\hline
			$|\B_h|$ & $1$ & $1$ & $6$ & $25$ & $199$ & $2769$ & $79229$ \\
			\hline
			$|\B_h|/|\T_h|$ & $0.5 $ & $0.25$ & $0.5$ & $0.446...$ & $0.436...$ & $0.402...$ & $0.413...$ \\ 
			\hline 
		\end{tabular}
\end{center}
\caption{The sizes of $\T_h$ and $\B_h$ and their ratio, for small $h$.}\label{table:3}
\end{table}

\section{Tournaments with a common subgraph}\label{sec:H*}
	
For a tournament $H^*$ with at most $h$ vertices,
let $\T_h(H^*)$ denote the set of all elements of $\T_h$ that contain $H^*$ as a sub-tournament.
Our main result in this section follows.
\begin{theorem}\label{t:sub}
	For all sufficiently large $h$, if $H^*$ contain at least $5\log h$ vertices then $\T_h(H^*)$ is dominant.
\end{theorem}
\begin{proof}
We assume that $h$ is sufficiently large and that the number of vertices of $H^*$
is $k$ where $h > k \ge 5\log h$.
We define a probability space of $n$-vertex tournaments (hereafter we assume that $n$ is a multiple of $h$, as this assumption does not affect the theorem's statement).
Assume that the vertices of $H^*$ are labeled with  $[k]$.
Consider vertex set $[n]$ partitioned into $k+1$ subsets $V_1,\ldots,V_{k+1}$.
For $i=1,\ldots,k$, set $V_i$ has $n/h$ vertices and set $V_{k+1}$ consists of the remaining $n-kn/h$ vertices.
For all $1 \le i < j \le k$, the edges between $V_i$ and $V_j$ are all directed from $V_i$ to $V_j$ if
$(i,j) \in E(H^*)$ or all directed from $V_j$ to $V_i$ if $(j,i) \in E(H^*)$.
Observe that each transversal of $V_1,\ldots,V_k$ induces a copy of $H^*$.
The remaining edges, which in particular include the edges having at least one endpoint in $V_{k+1}$, are oriented
randomly, uniformly and independently. Denote the resulting probability space by $T(n,h,H^*)$.
We prove that for a small positive $\beta=\beta(h)$ it holds that
for each $H \in \T_h(H^*)$, its expected density in $G \sim T(n,h,H^*)$ is at least $(1+\beta)d(H)$.
By the second moment method, exactly as in the proof of Lemma \ref{l:concentration}, this implies that
$\T_h(H^*)$ is dominant.

Let, therefore, $H \in \T_h(H^*)$ be labeled with vertex set $[h]$
such that the sub-tournament of $H$ induced by $[k]$ is label-isomorphic to $H^*$.
Recall that $d(H) = h!2^{-\binom{h}{2}}/aut(H)$.
Let $P$ denote the set of all $(h-k)!$ permutations of $[h]$ that are stationary on $[k]$,
let $Aut(H)$ denote the automorphism group of $H$ 
and let $Q \le Aut(H) \cap P$ be the sub-group of $Aut(H)$ consisting of the
permutations of $[h]$ that are stationary on $[k]$. Observe that $1 \le |Q| \le aut(H)$.

Suppose now that $G \sim T(n,h,H^*)$.
Consider a random injection $f$ from $[h]$ to $[n]$. We call $f$ {\em good} if
$f(i) \in V_i$ for $i=1,\ldots,k$ and $f(i) \in V_{k+1}$ for $i=k+1,\ldots,h$.
By the sizes of the $V_i$'s we have that $f$ is good with probability
$$
\frac{1}{h^k} \Pi_{i=k+1}^h\left(\frac{n-kn/h-i+k+1}{n-i+1}\right) \ge \frac{1}{(eh)^k}\;.
$$
Given that $f$ is good, the probability that its image induces a copy of $H$ is
$$
\frac{(h-k)!}{|Q|} 2^{\binom{k}{2}-\binom{h}{2}}
$$
since $\binom{h}{2}-\binom{k}{2}$ is the number of edges with an endpoint in $V_{k+1}$.
Hence, ${\mathbb E}[d_H(G)]$ (the expectation of $d_H(G)$) satisfies
\begin{align*}
{\mathbb E}[d_H(G)] & \ge \frac{1}{(eh)^k} \boldsymbol{\cdot} \frac{(h-k)!}{|Q|} 2^{\binom{k}{2}-\binom{h}{2}}\\
& \ge \frac{1}{(eh^2)^k} \boldsymbol{\cdot} \frac{h!2^{-\binom{h}{2}}}{aut(H)} 2^{\binom{k}{2}}\\
& = d(H) \frac{1}{(eh^2)^k} 2^{\binom{k}{2}}\\
\end{align*}
So, to prove the existence of $\beta=\beta(h)$ it suffices to prove that
$2^{(k-1)/2} > eh^2$. Indeed this holds as $k \ge 5 \log h$ and because $h$ is sufficiently large.
\end{proof}

\section{Concluding remarks and some open problems}\label{sec:concluding}
	
We list a few open problems and conjectures concerning dominant families.
An $h$-vertex tournament $H$ is {\em highly dominant} if every maximal dominant subset of $\T_h$ contains $H$.
The proposition shows that there are highly dominant tournaments.
\begin{proposition}\label{prop:1}
$T_h$ is highly dominant for all $h \ge 3$.	
\end{proposition}
\begin{proof}
	Let $\F \subset \T_h$ be dominant. Hence, there exists $\beta=\beta(\F)$ and $n_0 \in \mathbb{N}$ such that for all $n \ge n_0$, there exists a tournament $G$ with $n$ vertices such that 
	$d_H(G) \ge (1+\beta)d(H)$ for each $H \in \F$.
	For $n \ge n_0$ let $G_n$ be tournament satisfying $d_H(G_n) \ge (1+\beta)d(H)$ for each $H \in \F$.
	Fix some $H \in \F$. As $d_H(G_n) \ge (1+\beta)d(H)$ for all $n \ge n_0$, it follows from the result of Chung and Graham \cite{Chung1991} that $\{G_n\}$ is {\em not} a quasi-random sequence, as it violates property
	$P_1(h)$ there. But on the other hand, it follows from exercise 10.44(b) of \cite{Lovasz2007} and also from  \cite{Coregliano2017} that $T_h$ is quasi-random forcing, implying that for our sequence, there exists
	$\epsilon > 0$ and $n_1 \ge n_0$ such that for all $n \ge n_1$, $d_{T_h}(G_n)  \ge (1+\epsilon)d(T_h)$.
	This implies that $\{T_h\} \cup \F$ is dominant.
\end{proof}
\begin{problem}
	Determine all highly dominant tournaments. In particular, are there non-transitive highly dominant tournaments?
\end{problem}

It is very easy to show that for every positive integer $k \ge 2$, there is a minimum integer $f(k)$ such that for all $h \ge f(k)$, every $k$-subset of $\T_h$ is dominant.
The following proposition gives an upper bound for $f(k)$.
\begin{proposition}
	$f(k) \le (1+o_k(1))\log k$.
\end{proposition}
\begin{proof}
	Fix $\epsilon > 0$ and assume throughout the proof that $k$ is sufficiently large. Let
	$h \ge (1+\epsilon)\log k$. Let $r=h\lceil \sqrt{hk} \rceil$. Consider a complete graph $M$ on $r$ vertices.
	Take $r/h$ pairwise vertex-disjoint copies of $K_h$ (namely, a $K_h$-factor of $M$), remove the edges of this factor
	from $M$ and repeat taking factors. After taking $t$ factors we have already taken $tr/h$ pairwise edge-disjoint copies of $K_h$
	and the spanning subgraph of $M$ consisting of the edges not yet taken is regular of degree
	$r-1-t(h-1)$. By the Hajnal-Szemer\'edi Theorem \cite{Hajnal1970} we can do so as long as
	$r-1-t(h-1) \ge r-r/h$ so we can have $t \ge r/h^2$.  Thus, we can find in $M$ at least $r^2/h^3 \ge k$
	pairwise edge-disjoint copies of $K_h$.
	Now suppose that the vertices of $M$ are $[r]$ and that a set of $k$ pairwise edge-disjoint copies of $K_h$
	in $M$ is $\R=\{X_1,\ldots,X_k\}$ and $V(X_i)=\{x_{i,1},\ldots,x_{i,h}\}$.
	
	Now suppose that $\F=\{H_1,\ldots,H_k\} \subset \T_h$. We must prove that $\F$ is dominant.
	We assume that the vertices of each $H_i$ are labeled with $[h]$.
	Suppose that $n$ is an integer multiple of $r$. Consider vertex sets $V_1,\ldots,V_r$
	each of size $n/r$. We construct a random tournament with $n$ vertices as follows.
	For each $i=1,\ldots,k$, and for each pair $j,j'$ of distinct indices from $[h]$, we orient all
	edges from $V_{x_{i,j}}$ to $V_{x_{i,j'}}$ if $(j,j') \in E(H_i)$ else we orient all
	edges from $V_{x_{i,j'}}$ to $V_{x_{i,j}}$ if $(j',j) \in E(H_i)$. Notice that the orientations are well-defined as the elements of $\R$ are pairwise edge-disjoint. The remaining edge of $G$ (those
	having two endpoints in the same part $V_i$ or those between $V_i$ and $V_j$ where $i,j$ are not both in some
	element of $\R$) are oriented arbitrarily.
	
	Fix some $H_i \in F$. Then, $d_H(G)$ is at least  the probability that a randomly chosen $h$-set of $G$ is
	a transversal of $V_{x_{i,1}},\ldots,V_{x_{i,h}}$, as any such transversal induces a copy of $H_i$ in $G$.
	But the probability that a randomly chosen $h$-set of $G$ is such is at least $h!/r^h$, so
	$d_H(G) \ge h!/r^h$. It therefore remains to prove that
	$$
	\frac {h!}{r^h} > d(H) = \frac{h!2^{-\binom{h}{2}}}{aut(H)}
	$$
	so it suffices to prove that $r^h < 2^{\binom{h}{2}}$ or, equivalently, $2r^2 < 2^h$.
	Indeed, this holds since $r = h\lceil \sqrt{hk} \rceil$ and since $h \ge (1+\epsilon)\log k$.
\end{proof}
\begin{problem}
	Determine some small values of $f(k)$. In particular, determine $f(2)$.
\end{problem}

Let $g(h)$ denote the maximum size of a dominant subset of $\T_h$. Of course, we do not expect to obtain an exact formula for $g(h)$, as there is no such exact formula for $|\T_h|$. But perhaps good asymptotic values could be of obtained.
\begin{problem}
	Provide good estimates for $g(h)$.
\end{problem}

\bibliographystyle{plain}

\bibliography{references}

\end{document}